\documentclass{birkjour}
\usepackage{subfigure}
\usepackage{verbatim}
 \newtheorem{thm}{Theorem}[section]
 
 \newtheorem{lemma}[thm]{Lemma}
 \newtheorem{Proposition}[thm]{Proposition}
 \theoremstyle{definition}
 
 \theoremstyle{remark}
 
 \newtheorem{example}{Example}
 \numberwithin{equation}{section}
 
 \newcommand{\R}{\mathbb{R}}

   \newcommand{\U}{\mathcal{U}}
    \newcommand{\V}{\mathcal{V}}

\begin{document}

%-------------------------------------------------------------------------
% editorial commands: to be inserted by the editorial office
%
%---------------------------------------------------------------------------
%Insert here the title, affiliations and abstract:
%

\title[Iteration of Involutes of Constant Width Curves in the Minkowski Plane]
 {Iteration of Involutes of Constant Width Curves in the Minkowski Plane}

%----------Author 1
\author[M.Craizer]{Marcos Craizer}

\address{%
Departamento de Matem\'{a}tica- PUC-Rio\br
Rio de Janeiro\br
Brazil}

\email{craizer@puc-rio.br}

\thanks{The author want to thank CNPq for financial support during the preparation of this manuscript.}
%----------classification, keywords, date


\date{September 30, 2013}
%----------additions
%%% ----------------------------------------------------------------------

\begin{abstract}
In this paper we study properties of the area evolute (AE) and the center symmetry set (CSS) of a convex planar curve $\gamma$.
The main tool is to define a Minkowski plane where $\gamma$ becomes a constant width curve. In this Minkowski plane, the CSS
is the evolute of $\gamma$ and the AE is an involute of the CSS. We prove that the AE is contained in the region bounded by the CSS and 
has smaller signed area.

The iteration of involutes generate a pair of sequences of constant width curves with respect to the Minkowski metric and its dual, respectively.
We prove that these sequences are converging to symmetric curves with the same center, which can be regarded as a central point of the curve  $\gamma$. 
\end{abstract}

%%% ----------------------------------------------------------------------
\maketitle
%%% ----------------------------------------------------------------------
%\tableofcontents

\section{Introduction}

Consider a smooth convex curve $\gamma$ in the plane, boundary of a strictly convex set $\Gamma$.  We call a diameter
any chord connecting points of $\gamma$ whose tangents are parallel. The set  of midpoints of the diameters
is called {\it area evolute} (AE), while the envelope of these diameters is called the {\it center symmetry set} (CSS). 
These two set sets describe the symmetries of $\gamma$ and have been extensively studied (\cite{Giblin08},\cite{Holtom99},\cite{Giblin-Zakalyukin},\cite{Janeczko96}). For a discrete version
of the AE and CSS, see \cite{Craizer13}.

A Minkowski plane is a $2$-dimensional vector space with a norm. This norm can be characterized by its unit ball $\U$, which is a convex symmetric set. We shall assume that its
boundary curve $u$ is smooth with strictly positive euclidean curvature. 
The $u$-evolute ${N_0}$ of a curve $\gamma$ is the envelope of the $u$-normal lines.  Any curve whose $u$-evolute is ${N_0}$ is called an $u$-involute of ${N_0}$.
One can easily verify that there is a one parameter family $\gamma_c$ of $u$-involutes of ${N_0}$, the $u$-equidistants of $\gamma$. 
One can find several properties of Minkowski evolutes in \cite{Tabach97}.  For some applications of Minkowski evolutes and equidistants in computer graphics, see \cite{Ait-Haddou00}. 

There is a particular choice of Minkowski metric $u=u(\gamma)$ that makes $\gamma$ a constant $u$-width curve. Take just $\U$ to 
be the central symmetrization of $\Gamma$, namely, $\U=\frac{1}{2}(\Gamma)+(-\Gamma))$ (see \cite{Chakerian66}). In fact, we can choose any unit ball homothetic to $\U$. 
In this Minkowski metric, the evolute ${N_0}$ coincides with the CSS and we shall write $N_0=CSS(\gamma)$. 
Moreover, the area evolute $M$ of $\gamma$ is an $u$-equidistant of $\gamma$ and thus we shall write $M={\mathcal Inv}(N_0)$. 

In Minkowski theory, the dual norm $\V$ plays an important r\^ole. We shall be interested in the one parameter family of $v$-involutes of $M$, where $v$ denotes the boundary of $\V$.
We obtain a family $\delta_d$ of
constant $v$-width curves whose CSS is $M$. The AE of $\delta_d$ is independent of the choice of $d$ and we shall denote it by $N={\mathcal Inv}(M)$.  

Denote by $\alpha$ the $u$-curvature of $M$ and by $\beta$ the $v$-curvature of $N$. It is an interesting fact that the derivative
of $\beta$ with respect to $v$-arc length of $u$ is exactly $\alpha$. Moreover, $\beta$ can be interpreted as follows: A diameter $l(\theta)$ divides the region 
bounded by an equidistants $\gamma_c$ into two parts of area $A_1(c,\theta)$ and $A_2(c,\theta)$. Then the area difference $A_1(c,\theta)-A_2(c,\theta)$ grows linearly 
with $c$ and the rate of growth is exactly $4\beta(\theta)$. 

We can interpret the sup norm $||\beta||_{\infty}$ of $\beta$ as an asymmetry measure of $\gamma$ equivalent to the one proposed in \cite{Groemer}. 
In fact, $\gamma=\gamma_{c_0}$, for some $c_0$, and  the asymmetry function of $\gamma$ defined in \cite{Groemer} is the maximum of the ratios $A_1(c_0,\theta)/A_2(c_0,\theta)$. This is equivalent to consider the asymmetry measure of $\gamma$ as the maximum of the differences $A_1(c_0,\theta)-A_2(c_0,\theta)$, which is 
$4c_0||\beta||_{\infty}$.

Based on the area difference property, we show that $N$ is contained in region $\overline{M}$ bounded by $M$. This fact can be re-phrased by saying that the AE of a convex curve is contained
in the region bounded by its CSS. Although this is not surprising, we are not aware of any published proof.

Since the curves $M$ and $N$ may have self-intersections, it is not easy to compare the areas bounded by them. Thus we consider a substitute for these areas, the "signed areas".
Using some Minkowski isoperimetric inequalities, it is not difficult to show that the mixed areas $A(M,M)$ and $A(N,N)$ are negative, being zero if and only if $M$ and $N$ reduce to points.
Thus we define the signed areas $SA(M)=-A(M,M)$ and $SA(N)=-A(N,N)$. With this definition, we can prove that the difference $SA(M)-SA(N)$ is exactly 
the integral of the square of the $v$-curvature $\beta$ of $N$ with respect to $u$-arc length of $v$.  

It seems natural to iterate the involutes. We obtain a pair of sequences $(M_i)$ and $(N_i)$ of curves defined indutively by $M_0=M$, $N_{i}={\mathcal Inv}(M_{i-1})$ and $M_{i}={\mathcal Inv}(N_i)$. We have that 
$\overline{M_{i-1}}\supset\overline{N_i}\supset\overline{M_i}$, the curvature
of $M_{i-1}$ is the derivative of the curvature $N_i$ and the curvature of $N_i$ is the derivative of the curvature of $M_{i}$. A similar iteration for fronts can be found in
\cite{Takahashi12}.
We shall prove here that the curves $N_i$ and $M_i$ are converging to a constant curve $O$ in the $C^{\infty}$ topology. The point $O=O(\gamma)$ can be regarded as a center of symmetry 
of $\gamma$ and thus we shall call it the central point of $\gamma$.

These iterates can also be seen in terms of the equidistants. Consider the sequences of convex curves $\gamma_i$, $c$-equidistants of $M_i$, and $\delta_i$, $d$-equidistants
of $N_i$, with $c$ and $d$ fixed. Then these curves are of constant width in the Minkowski planes  $\U$ and $\V$, respectively. Moreover,  $\gamma_i$ and $\delta_i$ are converging 
in the $C^{\infty}$ topology to $O+cu$ and $O+dv$.

The paper is organized as follows: In section 2, we review concepts of Minkowski planar geometry, like arc-length, curvature, evolutes and involutes. In section 3 
we define the Minkowski metric such that $\gamma$ becomes of constant width. Then we describe the concepts of section 2 in the particular case of constant width curves. In section 4 we prove two results
comparing the AE and CSS of a curve. Finally in section 5 we prove the convergence of the iteration of involutes to a constant.

\section{Curves in a Minkowski plane}

In this section, we describe the basic definitions and properties of a Minkowski norm in the plane. For details, see \cite{Thompson96}.

We denote by $[w_1,w_2]$ the determinant of the $2\times 2$ matrix whose columns are $w_1$ and $w_2$. 
Along the paper, unless otherwise stated, symmetry will always mean symmetry with respect to the origin.

\subsection{Minkowski plane and its dual}

Consider a convex symmetric set $\U\subset\R^2$ . For any $X\in\R^2$, write $X=tu$, for some $t\geq 0$ and $u$ in the boundary $\U$. Then $||X||_{u}=t$
is a Minkowski norm in the plane. We shall assume that $\U$ is strictly convex and its boundary $u$ is a smooth curve.

Denoting $e_r=(\cos(\theta),\sin(\theta))$ and $e_{\theta}=(-\sin(\theta),\cos(\theta))$, parameterize $u$ by $u(\theta)$, $0\leq\theta\leq 2\pi$, such that 
$u'(\theta)$ is a non-negative multiple of $e_{\theta}$. We can write
$$
u(\theta)=a(\theta)e_r+a'(\theta)e_{\theta},
$$
where $a(\theta)$ is the support function of $\U$. We shall assume that $(a+a'')(\theta)>0$,  for any $0\leq\theta\leq 2\pi$, which is equivalent to say that the curvature 
of $u$ is strictly positive.

The dual unit ball $\U^*$ can be identified with a convex set $\V$ in the plane by  $u^*(w)=[w,v]$, for any $w\in\R^2$. 
Define
\begin{equation}\label{eq:definev}
v(\theta)=\frac{u'(\theta)}{[u(\theta),u'(\theta)]}.
\end{equation}
Since $[u,v]=1$ and $[u',v]=0$, $v(\theta)$ is a parameterization of the boundary of $\V$. 
It is not difficult to verify that $v$ is a convex symmetric curve with strictly positive curvature. Moreover, 
\begin{equation}\label{eq:dual}
u(\theta)=-\frac{v'(\theta)}{[v(\theta),v'(\theta)]}. 
\end{equation}

\subsection{ Minkowski length and curvature}

Given a smooth curve $\gamma$, parameterize it such that $\gamma'(\theta)$ is a non-negative multiple of $e_{\theta}$ and write
\begin{equation}\label{eq:parametergamma}
\gamma'(\theta)=\lambda(\theta)v(\theta), 
\end{equation}
$\lambda(\theta)\geq 0$, $a\leq\theta\leq b$. The Minkowski $v$-length $L_v$  of $\gamma$ is defined as 
$$
L_v(\gamma)=\int_a^b \lambda(\theta) d\theta,
$$ 
(see \cite{Thompson96}). 
The Minkowski normal line at $\gamma(\theta)$ is defined as $\gamma(\theta)+su(\theta)$, $s\in\R$. 
The Minkowski center of curvature $C$ and  the Minkowski curvature radius $R$ of $\gamma$ at $\gamma(\theta)$ are defined by the 
condition that the contact of $C+Ru$ and $\gamma$ at $\gamma(\theta)$ is of order $3$
(\cite{Tabach97}).

\begin{lemma}\label{lemma:curvature}
Consider a curve $\gamma$ satisfying equation \eqref{eq:parametergamma}. 
The Minkowski center of curvature $C$ lies in the Minkowski normal and the Minkowski radius of curvature is ${\mu}(\theta)$, where $\lambda(\theta)={\mu(\theta)}[u,u'](\theta)$. 
\end{lemma}

\begin{proof}
Let $F(X)=||X-C||_u-R$. Then $F(C+Ru)=0$ and thus $DF(C+Ru)\cdot u'=0$. Differentiating this equation and using $u''=[u,u'']v+[u,u']v'$ we obtain
\begin{equation}\label{eq:ProofCurvature}
RD^2F(C+Ru)\cdot(u',u')+[u,u']DF(C+Ru)\cdot v'=0. 
\end{equation}
Now let $f(\theta)=F(\gamma(\theta))$. Then $f'(\theta)=DF(\gamma(\theta))\cdot \gamma'(\theta)$ and so $\gamma$ has contact of order $2$ with $C+Ru$ if and only if $\gamma(\theta)=C+Ru(\theta)$.
Moreover
$$
f''(\theta)=\mu^2D^2F(\gamma(\theta))\cdot(u'(\theta),u'(\theta))+\mu [u,u'] DF(\gamma(\theta))\cdot v'.
$$
We conclude from equation \eqref{eq:ProofCurvature} that $\gamma$ has contact of order $3$ with $C+Ru$ if and only if $R=\mu$. 
\end{proof}

\subsection{Minkowski evolutes, involutes and equidistants}

Define the $u$-evolute of a curve $\gamma$ as the envelope of its $u$-normals (\cite{Tabach97}). Then lemma \ref{lemma:curvature} implies that
$$
N_0(\theta)=\gamma(\theta)-{\mu}(\theta)u(\theta)
$$

For a fixed $c$, the curves
$$
\gamma_c(\theta)=\gamma(\theta)+cu(\theta)
$$
are called the $u$-equidistants of $\gamma$.

\begin{lemma}\label{lemma:involute}
The equidistants of $\gamma$ have the same evolute as $\gamma$. Reciprocally, if $\gamma_1$ has $N_0$ as its evolute,
then $\gamma_1$ is an equidistant of $\gamma$.
\end{lemma}

\begin{proof}
The evolute of  $\gamma_c(\theta)$ is given by
$$
\gamma(\theta)+cu(\theta)- (\mu(\theta)+c) u(\theta)= N_0(\theta).
$$
Reciprocally, if the evolute of $\gamma_1$ is $N_0$ then
$$
\gamma(\theta)-\gamma_1(\theta)=({ \mu}(\theta)-{ \mu}_1 (\theta)) u(\theta).
$$
Differentiating we obtain ${ \mu}(\theta)-{ \mu}_1(\theta)=-c$, which proves the lemma. 
\end{proof}

The $u$-involute of $N_0$ is any curve whose $u$-evolute is $N_0$. By lemma \ref{lemma:involute}, there exists a one parameter family of $u$-involutes of $N_0$, namely, the $u$-equidistants
of $\gamma$. 

\subsection{ Mixed area and an isoperimetric inequality}

Assume now that $\gamma$ is the smooth boundary of a convex region $\Gamma$. Parameterize $\gamma$ satisfying equation \eqref{eq:parametergamma}, with $0\leq\theta\leq 2\pi$. 

The mixed area of $\gamma$ and $u$ is given by
$$
A(\gamma,u)=\frac{1}{2}\int_0^{2\pi}[u,\gamma']d\theta=\frac{1}{2}\int_0^{2\pi}\lambda(\theta)d\theta=\frac{1}{2} L_v(\gamma),
$$
where $L_v$ denotes the $\V$-length of $\gamma$. Denoting by $A(\gamma)$ and $A(u)$ the areas of $\Gamma$ and $\U$, 
the Minkowski inequality 
$$
A(\gamma,u)^2\geq A(\gamma)A(u)
$$
implies that
\begin{equation}\label{ineq:isoperimetric}
L_v^2\geq 4A(\gamma)A(u),
\end{equation}
with equality if and only if $\gamma$ is homothetic to $u$. For more details, see \cite{Thompson96}, ch.4.

\section{Constant width curves in the Minkowski plane}

From now on, $\gamma$ denote a smooth curve, boundary of a convex planar region $\Gamma$. Consider a parameterization $\gamma(\theta)$, $0\leq\theta\leq 2\pi$, of $\gamma$
$\gamma'(\theta)$ is a non-negative multiple of $e_{\theta}$. We shall assume that the curvature of $\gamma$ is strictly positive, for any $0\leq\theta\leq 2\pi$.  

\subsection{Minkowski metric associated with a convex curve}

In this section we shall define a Minkowski metric $u$ such that $\gamma$ becomes a constant $u$-width curve. 

Given $\gamma$ as above, its area evolute $M$ is given by
\begin{equation}\label{eq:defineM}
M(\theta)=\frac{1}{2}\left( \gamma(\theta)+\gamma(\theta+\pi) \right).
\end{equation}

When the diameters of $\gamma_1(\theta)$ are parallel to the diameters of $\gamma_2(\theta)$, for any $0\leq\theta\leq 2\pi$, we shall simply say that  $\gamma_1$ and $\gamma_2$ are parallel .
Defining $u(\gamma)$ by
\begin{equation}\label{eq:defineM}
u(\gamma)(\theta)=\frac{1}{2}\left( \gamma(\theta)-\gamma(\theta+\pi) \right),
\end{equation}
we obtain that $u(\gamma)$ is symmetric and parallel to $\gamma$. One can also verify easily that the curvature of $u$ is strictly positive.

Next lemma says that $\gamma$ is parallel to a symmetric curve $u$ if and only if $u$ is homothetic to $u(\gamma)$. 

\begin{lemma}
Up to homothety, $u(\gamma)$ is the only symmetric curve parallel to $\gamma$.
\end{lemma}
\begin{proof}
Take $u$ any symmetric curve parallel to $\gamma$.  We can write $\gamma(\theta)=M(\theta)+c(\theta)u(\theta)$, for some $c(\theta)$. 
Now $\gamma'(\theta)=M'(\theta)+c(\theta)u'(\theta)+c'(\theta)u(\theta)$, which implies that $c(\theta)$ is a constant. We conclude that $u$ is homothetic to $u(\gamma)$.
\end{proof}

From now on, we shall denote simply by $u$ any curve homothetic to $u(\gamma)$. 
For a fixed $c$, define the  {\it equidistant} of $\gamma$ at level $c$ with respect to $u$ by
\begin{equation}\label{eq:equi}
\gamma_c(\theta)=M(\theta)+c u(\theta),
\end{equation}
Thus we have a one-parameter family of equidistants that includes $\gamma$ and the $0$-equidistant $M$. 
It is not difficult to verify that 
two curves $\gamma_1$ and $\gamma_2$ are equidistants if and only if $M(\gamma_1)=M(\gamma_2)$ and $u(\gamma_1)$ is homothetic to $u(\gamma_2)$.

We shall consider the Minkowski plane with the metric defined by $u$. 
We say that $\gamma$ has {\it constant $u$-width} if $\gamma(\theta)-\gamma(\theta+\pi)=2cu(\theta)$, for some constant $c$.  

\begin{lemma}\label{lemma:gammaconstantwidth}
$\gamma$ has constant $u$-width if and only if $\gamma$ and $u$ are parallel.
\end{lemma}
\begin{proof}
If $\gamma$ is parallel to $u$, then it is given by equation \eqref{eq:equi} for some $M$. Thus $\gamma$ has constant $u$-width. 
Reciprocally, write $\gamma(\theta)=M(\theta)+c_1u(\gamma)(\theta)$. If $\gamma$ has constant $u$-width, 
$2c_1u(\gamma)(\theta)=2c u(\theta)$, for some $c$. Thus
$u(\gamma)$ is homothetic to $u$ and hence $\gamma$ is parallel to $u$.
\end{proof}

From lemma \ref{lemma:gammaconstantwidth}, we conclude that $\gamma$ has constant $u$-width. For more details of this section, see \cite{Chakerian83}. 

\subsection{Cusps of the equidistants}

Denote by $\alpha(\theta)$ the $u$-curvature radius of $M$ at $M(\theta)$. By lemma \ref{lemma:curvature}, we can write
\begin{equation}\label{eq:definealpha}
M'(\theta)=\alpha(\theta)u'(\theta),
\end{equation}
Then the $u$-curvature radius of the equidistant $\gamma_c$ is $\alpha+c$, i.e.,
\begin{equation*}
\gamma_c'(\theta)=(\alpha(\theta)+c)u'(\theta).
\end{equation*}  

The cusps of $\gamma_c$ corresponds to points where $\alpha+c$ is changing sign.  
Observe that
\begin{equation*}
\gamma_c''=(\alpha+c)'u'+(\alpha+c)u''
\end{equation*}
and so
\begin{equation*}
[\gamma_c',\gamma_c'']=(\alpha+c)^2[u',u''].
\end{equation*}
We conclude that $\gamma_c$ is convex outside cusps. In particular, $\gamma_c$ is convex  for $c\geq ||\alpha||_{\infty}$, 
where
$$
||\alpha||_{\infty}= \sup_{\theta\in[0,\pi]}|\alpha(\theta)|.
$$

Next proposition says that the number of cusps of $M$ is odd and at least three. A proof of this fact can be found in \cite{Giblin08}.
We give another proof here for the sake of completeness. 

\begin{Proposition}\label{prop:cuspsM}
The number of cusps of $M$ is odd and bigger than or equal to three. 
\end{Proposition}

\begin{proof}
Write $\gamma'(\theta)=\lambda(\theta) v(\theta)$, $\lambda>0$. We look for zeros of $\Lambda(\theta)=\lambda(\theta+\pi)-\lambda(\theta)$. 
Since $\Lambda(\theta+\pi)=-\Lambda(\theta)$, the number of zeros is odd. We have to verify that this number cannot be one.
We may assume that $\Lambda(0)=0$. Then the horizontal distance between $\gamma(0)$ and $\gamma(\pi)$ is $\int_0^{\pi}\gamma_1'(\theta+\pi)d\theta$
and also $-\int_0^{\pi}\gamma_1'(\theta)d\theta$, where $\gamma_1'$ denotes the horizontal component of $\gamma'$. Thus
$$
\int_0^{\pi}\lambda(\theta+\pi)v_1(\theta)d\theta-\int_0^{\pi}\lambda(\theta)v_1(\theta) d\theta=0,
$$
where $v_1$ denotes the horizontal component of $v$. From this equation it follows that $\Delta(\theta)=0$ at least once in the interval $(0,\pi)$. 
\end{proof}

\subsection{ Barbier's theorem}

For $c\geq ||\alpha||_{\infty}$, the curve $\gamma_c$ is convex. Then the $\V$ length $L_v$ of $\gamma_c$ is
$$
L_v=\int_{\theta=0}^{2\pi} (\alpha+c)[u,u']d\theta= 2A(u)c,
$$
where the last equality comes from $\alpha(\theta+\pi)=-\alpha(\theta)$. In the Euclidean case, this result is known as Barbier's theorem. Observe that it can also be written as
\begin{equation}\label{eq:Barbier}
A(\gamma_c,u)=A(u)c. 
\end{equation}

If we admit signed lengths, Barbier's theorem can be extended to equidistants with cusps. 
In particular, the signed $v$-length of $M$ is zero. In fact, this last result holds not only for constant width curves, but for any smooth closed convex curve (\cite{Tabach97}). 

\subsection{ Signed area of the area evolute}\label{sec:signedareas}

Consider a convex constant $u$-width curve $\gamma_c$. From Barbier's theorem, the isoperimetric inequality \eqref{ineq:isoperimetric} can be written as 
\begin{equation}\label{eq:Iso1}
A(\gamma_c)\leq c^2A(u),
\end{equation}
with equality only for $M=0$. 
This result can also be extended to non-convex equidistants by considering mixed areas. 
\begin{lemma}
For any equidistant $\gamma_c$ we have 
\begin{equation*}
A(\gamma_c,\gamma_c)\leq c^2A(u),
\end{equation*}
with equality if and only if $M=0$. 
\end{lemma}
\begin{proof}
Write $\gamma_{c}=\gamma_{c_1}+(c-c_1)u$, for some $c_1$ with $\gamma_{c_1}$ convex.  Then
$$
A(\gamma_{c},\gamma_{c})=A(\gamma_{c_1})+2(c-c_1)A(\gamma_{c_1},u)+(c-c_1)^2A(u)
$$
Using equation \eqref{eq:Barbier} we obtain 
$$
A(\gamma_{c},\gamma_{c})=A(\gamma_{c_1})-(c_1^2-c^2)A(u).
$$
Using now equation \eqref{eq:Iso1} we conclude the proof.
\end{proof}

In particular, the mixed area $A(M,M)$ of $M$ is non-positive and equals zero only for $M=0$. We define the signed area of $M$ a
by 
\begin{equation}
SA(M)=-A(M,M).
\end{equation}
Thus $SA(M)$ is non-negative and equal $0$ only for $M=0$. When $M$ has no self-intersections, it corresponds to the area of the region bounded by $M$. 

\subsection{ Involutes of constant width curves }

In this section we define the $v$-involute $N$ of $M$. 

Let
\begin{equation}\label{eq:definebeta}
\beta(\theta)=\frac{1}{2}\int_{\theta}^{\theta+\pi}\alpha(s)[u,u'](s)ds
\end{equation}
where $\alpha$ is defined by equation \eqref{eq:definealpha}, and define 
\begin{equation}\label{eq:defineinvoluta}
N(\theta)=M(\theta)+\beta(\theta)v(\theta).
\end{equation}
Let $\eta_d$ denote the one parameter family of $v$-equidistants of $N$, i.e.,
\begin{equation}\label{eq:equidistantsN}
\eta_d(\theta)=N(\theta)+dv(\theta).
\end{equation}

\begin{lemma}
We have that
\begin{equation}\label{eq:derivativeN}
N'(\theta)=\beta(\theta)v'(\theta)
\end{equation}
and the curves $\eta_d$ are constant $v$-width curves with $v$-curvature radius $\beta+d$. Moreover
for any $d$, the $v$-evolute of $\eta_d$ is $M$. 
\end{lemma}
\begin{proof}
Observe that
$$
N'(\theta)=\alpha(\theta)u'-\alpha(\theta)[u,u']v+\beta(\theta)v'=\beta(\theta)v'.
$$
which proves equation \eqref{eq:derivativeN}. This equation implies that the curves $\eta_d$ are constant $v$-width curves and that 
the $v$-curvature radius of $\eta_d$ is $\beta+d$. Finally, the evolute of $\eta_d$ is given by
$$
\eta_d(\theta)-(\beta(\theta)+d)v(\theta)=N(\theta)-\beta(\theta)v(\theta)=M(\theta).
$$
which completes the proof of the lemma.
\end{proof}

We conclude from the above lemma that $\eta_d$ is an involute of $M$, for any $d$, and we shall write $N={\mathcal Inv}(M)$. 

If $d\geq ||\beta||_{\infty}$, the equidistant $\eta_d$ is convex. 
Since $N$ is the AE of $\eta_d$, by proposition \ref{prop:cuspsM}, it has an odd number of cusps, at least three. In fact, 
$N$ has at most the same number of cusps as  $M$. 

\begin{lemma}
$M$ has at least the same number of cusps as $N$.
\end{lemma}
\begin{proof}
The cusps of $N$ are zeros of $\beta:[0,2\pi]\to\R$, while the cusps of $M$ occur when $\beta'(\theta)$ is changing sign. 
Since between two zeros of $\beta$ there is at least one change of sign of $\beta'$, the lemma is proved.
\end{proof}

\bigskip\bigskip

\begin{example}\label{ex:example1}
In the euclidean plane, let $u=e_r=(\cos(\theta),\sin(\theta))$ and $v=e_{\theta}=(-\sin(\theta),\cos(\theta))$. Let $\gamma_c(\theta)=M(\theta)+c e_r$, where 
$$
M(\theta)=\left( 2\sin(2\theta)-\sin(4\theta), 2\cos(2\theta)+\cos(4\theta)   \right).
$$
Straightforward calculations shows that
$$
M'(\theta)=-8\sin(3\theta) e_{\theta},
$$
and so $\alpha(\theta)=-8\sin(3\theta)$. From equation \eqref{eq:definebeta} we obtain $\beta(\theta)=-\frac{8}{3}\cos(3\theta)$ and hence, by equation \eqref{eq:defineinvoluta}, 
$$
N(\theta)= \left(    \frac{2}{3}\sin(2\theta)+\frac{1}{3}\sin(4\theta) , \frac{2}{3}\cos(2\theta)-\frac{1}{3}\cos(4\theta)    \right).
$$
Finally, 
$$
\int_0^{\pi} [M,M']d\theta=-4\pi, \ \ \ \int_0^{\pi} [N,N']d\theta=-\frac{4\pi}{9},
$$
which implies that the signed areas of $M$ and $N$ equals $4\pi$ and $\frac{4\pi}{9}$, respectively. Observe that since $M$ and $N$ have not self-intersections, 
the signed areas are in fact the areas of the regions bounded by $M$ and $N$. 

\begin{figure}[htb]
 \centering
 \includegraphics[width=0.50\linewidth]{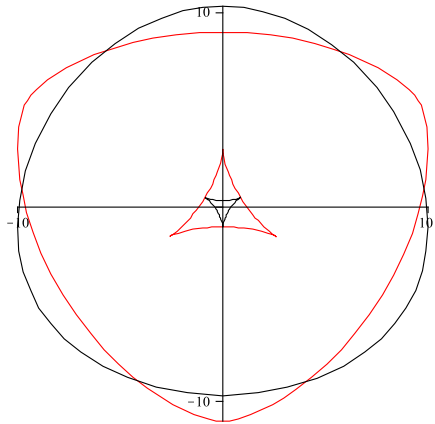}
 \caption{ The curves $M$ and $N$ with equidistants at level $c=10$. }
\label{fig:deltoide}
\end{figure}
\end{example}

\subsection{Rate of growth of the area difference}

Let $\gamma_c^1(\theta)$ denote the curve $\gamma_c(s),\ \theta\leq s\leq \theta+\pi$ and denote by  $A_1(c,\theta)$ the area of the region bounded by $\gamma_c^1(\theta)$ 
and the diameter $l(\theta)$. Let $A_2(c,\theta)=A(\gamma_c)-A_1(c,\theta)$.
Next proposition says that the growth of the area difference $A_1-A_2$ is linear with $c$ with rate $4\beta$. 
This result can be found in \cite{Chakerian83}, but we give a proof for the sake of completeness. 

\begin{Proposition}\label{prop:linear}
We have that
\begin{equation}\label{eq:linear}
A_1(c,\theta)-A_2(c,\theta)= 4c\beta(\theta).
\end{equation}
\end{Proposition}

\begin{proof}
We have that 
$$
2A_1(\theta)=\int_{\theta}^{\theta+\pi}[\gamma(s)-M(\theta),\gamma'(s)] ds
$$
$$
=\int_{\theta}^{\theta+\pi}\left[cu(s)+M(s)-M(\theta), cu'(s)+M'(s) \right] ds.
$$
The area $A_2(\theta)$ is obtained by integrating the same integrand from $\theta+\pi$ to $\theta+2\pi$. Thus
$$
(A_1-A_2)(\theta)=\int_{\theta}^{\theta+\pi}\left[ M(s)-M(\theta), cu'(s)\right]- \left[ M'(s), cu(s)\right] ds.
$$
Since
$$
\int_{\theta}^{\theta+\pi}\left[ M(s)-M(\theta), cu'(s)\right] ds=- \int_{\theta}^{\theta+\pi} \left[ M'(s), cu(s)\right] ds,
$$
we conclude that 
\begin{equation*}
A_1(\theta)-A_2(\theta)=-2c\int_{\theta}^{\theta+\pi} [M'(s),u(s)] ds.
\end{equation*}
Hence
$$
A_1(\theta)-A_2(\theta)=2c\int_{\theta}^{\theta+\pi} \alpha[u, u'](s) ds=4c \beta(\theta), 
$$
thus proving the proposition. 
\end{proof}

\section{Some relations between the area evolute and the center symmetry set}

For the family $\eta_d$, $N$ is the area evolute and $M$ the center symmetry set.
In this section we compare the signed areas of $M$ and $N$ and prove that $N$ is contained in the interior of $M$.

\subsection{Relation between signed areas of $M$ and $N$}

Recall that the signed area of $M$ is non-negative, being zero only if $M=0$. Of course, the same holds for $N$. 

\begin{Proposition}\label{prop:signedMN}
Denoting by $SA(M)$ and $SA(N)$ the signed areas of $M$ and $N$, we have
$$
SA(M)-SA(N)=\int_0^{\pi}\beta^2 \left[ v,v' \right]d\theta. 
$$
\end{Proposition}
\begin{proof}
Observe that
$$
\left[ M,M'  \right]=  \left[ N-\beta v, \alpha u' \right]= \alpha [N,u']=-\beta'[N,v],\ \ \left[  N,N' \right] =\beta \left[ N, v' \right]
$$
and so
$$
-\left[ M,M'  \right]+\left[  N,N' \right] =[N, (\beta v)']. 
$$
Thus 
$$
SA(M)-SA(N)= -\int_0^{\pi}  \left[ M,M'  \right]d\theta+\int_0^{\pi}\left[  N,N' \right] d\theta=
$$
$$
=-\int_0^{\pi}\beta \left[ N', v \right]d\theta=\int_0^{\pi}\beta^2 \left[ v,v' \right]d\theta.
$$
\end{proof}

\subsection{ Relative position of $M$ and $N$}

We prove now that the area evolute of a convex curve is contained in the region bounded by its
center symmetry set. This is not a surprising result, but we are not aware of any published proof of it.

The exterior of the curve $M$ is defined as the set of points of the plane that can be reached from a point of $\gamma$ by a path that do not cross $M$.
The region $\overline{M}$ bounded by $M$ is the complement of its exterior. It is well known that a point in the exterior of $M$ is the center of exactly one chord of $\gamma$ (\cite{Giblin08}). 

In this section we prove the following result:

\begin{Proposition}\label{prop:NsubsetM}
The involute $N$ is contained in the region bounded by $M$.
\end{Proposition}

The proof is based on two lemmas. For a fixed $\theta$, take $C=M(\theta)+sv(\theta)$, for some $s$, and denote by $l(s)$ the line parallel to $l(\theta)$ through $C$, where $l(\theta)$ is the
diameter line through $\gamma(\theta)$ and $\gamma(\theta+\pi)$. 
Then $l(s)$ divide the interior of $\gamma$ into two regions of areas $B_1=B_1(\theta,C)$ and $B_2=B_2(\theta,C)$.

\begin{lemma}
Take $C=N$, i.e., $s=\beta(\theta)$. Then $B_1(\theta,N)>B_2(\theta,N)$.
\end{lemma}
\begin{proof}
We have that
$$
B_1(\theta,N)=A_1(\theta)-(2c\beta-\delta),\ \ B_2(\theta,N)=A_2(\theta)+(2c\beta-\delta),
$$
where $\delta$ is the area of the regions outside $\gamma$ and between $l(\beta)$, $l(\theta)$ and the tangents to $\gamma$ at $\theta$ and $\theta+\pi$ (see figure \ref{fig:sym5}). Since, 
by proposition \ref{prop:linear}, $4c\beta=A_1-A_2$ we conclude that
$$
B_1(\theta,N)=\frac{A(\gamma)}{2}+\delta,\ \ B_2(\theta,N)=\frac{A(\gamma)}{2}-\delta, 
$$
which proves the lemma.
\end{proof}

\begin{figure}[htb]
 \centering
 \includegraphics[width=0.70\linewidth]{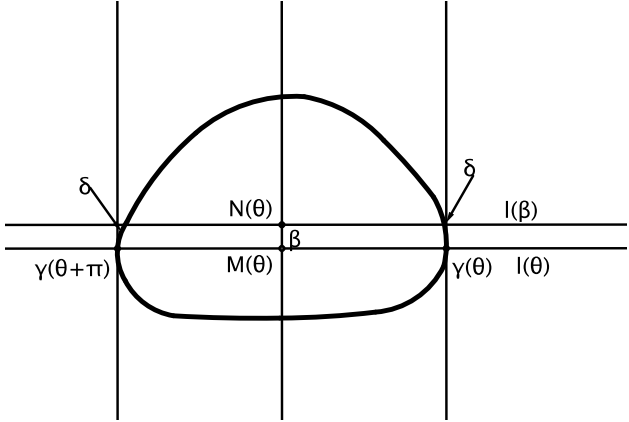}
 \caption{  The line $l(\beta)$ divides the interior of $\gamma$ into two regions of areas $B_1$ and  $B_2$.}
\label{fig:sym5}
\end{figure}

\begin{lemma}\label{lemma:interiorM}
If $B_1(C)\geq B_2(C)$ then $C$ is in the region bounded by $M$.
\end{lemma}
\begin{proof}
By an affine transformation of the plane, we may assume that $l(\theta)$ and $\gamma'(\theta)$ are orthogonal.
Consider polar coordinates $(r,\phi)$ with center $C$ and describe $\gamma$ by $r(\phi)$. Assume that $\phi=0$ at the line $l(s)$ and that $\phi=-\phi_0$ at $\gamma(\theta)$ (see figure \ref{fig:sym6}).
Denote the area of the sector bounded by $\gamma$ and the rays $\phi_1,\phi_2$ by 
$$
A(\phi_1,\phi_2)=\frac{1}{2}\int_{\phi_1}^{\phi_2}r^2(\phi)d\phi.
$$

Consider a line $q$ parallel to $\gamma'(\theta)$ through the point  $Q_0$ of $\gamma$ corresponding to $\phi=\pi$ and denote by $Q_1$ and $Q_2$ its intersection with the rays $\phi=\pi-\phi_0$
and $\phi=\pi+\phi_0$, respectively (see figure \ref{fig:sym6}). By convexity, we have that 
$$
A(\pi-\phi_0,\pi)<A(CQ_0Q_1)=A(CQ_0Q_2)<A(\pi,\pi+\phi_0).
$$
A similar reasoning shows that 
$ A(0,\phi_0)<A(2\pi-\phi_0,2\pi)$.

Now if $r(\phi+\pi)>r(\phi)$ for any $\phi_0<\phi<\pi-\phi_0$, we would have $B_1(C)<B_2(C)$, contradicting the hypothesis. We conclude that
$r(\phi+\pi)=r(\phi)$ for at least two values of $\phi_0<\phi<\pi-\phi_0$. Since the equality holds also for some $\pi-\phi_0<\phi<\pi+\phi_0$,
there are at least three chords of $\gamma$ having $C$ as midpoint. Thus $C$ is contained in the region bounded by $M$. 
\end{proof}

\begin{figure}[htb]
 \centering
 \includegraphics[width=0.70\linewidth]{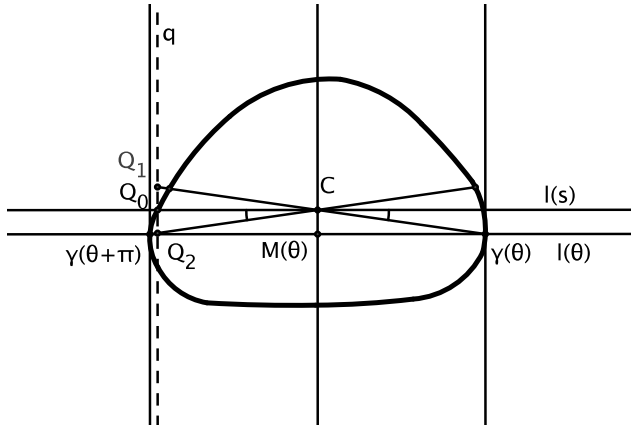}
 \caption{ Proof of lemma \ref{lemma:interiorM}: The line $q$ and the angles corresponding to $\phi=\pi+\phi_0$ and $\phi=2\pi-\phi_0$.}
\label{fig:sym6}
\end{figure}

\section{Iterating involutes}

We denote $M_0=M$ and $N_1=N={\mathcal Inv}(M)$. Let $M_{i}={\mathcal Inv}(N_i)$, $N_{i+1}={\mathcal Inv}(M_{i})$. 
Consider also the sequence
of smooth functions $\alpha_i,\beta_i:[0,2\pi]\to\R$ defined by 
\begin{equation}\label{eq:definealphaibetai}
M_i'=\alpha_i u',\ \ N_i'=\beta_i v' .
\end{equation}
It follows from equation \eqref{eq:definebeta} that
\begin{equation}\label{eq:derivalphaibetai}
\beta_{i+1}'=-\alpha_i [u,u'] , \ \  \alpha_{i}'=\beta_i[v,v'], 
\end{equation}
or equivalently,
\begin{equation}\label{eq:integraalphaibetai}
\beta_{i+1}=\frac{1}{2}\int_{\theta}^{\theta+\pi}\alpha_i[u,u']ds, \ \ \alpha_{i}=-\frac{1}{2}\int_{\theta}^{\theta+\pi}\beta_i[v,v']ds. 
\end{equation}
Also, from \eqref{eq:defineinvoluta}, 
\begin{equation}\label{eq:NiMi}
N_{i+1}=M_i+\beta_{i+1}v,\ \ M_{i}=N_i+\alpha_{i}u.
\end{equation}
From proposition \ref{prop:NsubsetM}
$$
{\overline M}_0\supset{\overline N}_1\supset{\overline M}_1\supset ...
$$
and we denote by $O(\gamma)$ the intersection of all these sets. 

The following theorems are the main results of the paper:

\begin{thm}\label{thm:ConvMiNi}
The set $O(\gamma)$ consists of a unique point and the curves $M_i$ and $N_i$ are converging to $O(\gamma)$ in the $C^{\infty}$ topology.
\end{thm}

We shall call $O=O(\gamma)$ the central point  of $\gamma$. For fixed $c$ and $d$ construct the sequences of convex curves
$$
\gamma_i=M_i+cu, \ \  \eta_i=N_i+dv.
$$
The curves $\gamma_i$ are of constant $u$-width while the curves $\eta_i$ are of constant $v$-width. 
We can re-state theorem \ref{thm:ConvMiNi} as follows:

\begin{thm}
The sequences of curves $\gamma_i$ and $\eta_i$ are converging in the $C^{\infty}$ topology to $O+cu$ and $O+dv$, respectively.  
\end{thm}

\bigskip

\begin{example}\label{ex:example2}
In example \ref{ex:example1}, we obtain from \eqref{eq:derivalphaibetai} that $\alpha_1(\theta)=-\frac{8}{9}\sin(3\theta)$. Then equation \eqref{eq:NiMi} implies that
$$
M_1(\theta)=N(\theta)+\alpha_1(\theta)e_r=\frac{1}{9}M(\theta).
$$
In fact, it is not difficult to verify that $M_{i+1}=\frac{1}{9}M_i$ and $N_{i+1}=\frac{1}{9}N_i$, for any $i\geq 0$. Thus the sequences $(M_i)$ and $(N_i)$ are both converging 
to $0$ in the $C^{\infty}$ topology (see figure \ref{fig:deltoide2}). 

 \begin{figure}[htb]
 \centering
 \includegraphics[width=0.40\linewidth]{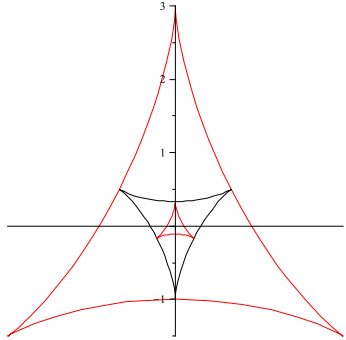}
 \caption{ The curves $M$, $N$ and $M_1$ of example \ref{ex:example2}. }
\label{fig:deltoide2}
\end{figure}
\end{example}

\bigskip\noindent
We shall now prove theorem \ref{thm:ConvMiNi}. Denote the signed areas of $M_i$ and $N_i$ by 
$$
SA(M_i)=-\int_0^{\pi}[M_i,M_i']d\theta, \ \ SA(N_i)=-\int_0^{\pi}[N_i,N_i']d\theta.
$$
By section \ref{sec:signedareas}, $SA(M_i)\geq 0$, $SA(N_i)\geq 0$ and proposition \ref{prop:signedMN} implies that
$$
SA(M_i)-SA(N_{i+1})=\int_0^{\pi} \beta_{i+1}^2[u,u']d\theta, \ \  SA(N_i)-SA(M_{i})=\int_0^{\pi} \alpha_{i}^2[v,v']d\theta.
$$
Thus 
\begin{equation}\label{eq:sumsquares}
\sum_{i=0}^{\infty} \int_0^{\pi} \beta_{i+1}^2[u,u']d\theta +\sum_{i=0}^{\infty} \int_0^{\pi} \alpha_{i}^2[v,v']d\theta \leq SA(M_0).
\end{equation}

We shall use below the following well-known inequality: For any continuous function $g:[0,\pi]\to\R$, 
\begin{equation}\label{eq:Holder}
\frac{1}{A(u)}\int_0^{\pi}|g|[u,u']d\theta\leq \left( \frac{1}{A(u)}\int_0^{\pi}g^2[u,u']d\theta      \right)^{\frac{1}{2}}.
\end{equation}

\begin{lemma}\label{lemma:boundsalphaibetai}
The functions $\alpha_i$ and $\beta_i$ are uniformly bounded in the $C^{\infty}$-topology. 
\end{lemma}

\begin{proof}
By equation \eqref{eq:sumsquares}, 
$$
\int_0^{\pi}\alpha_i^2 [u,u']d\theta\leq SA(M_0).
$$
Now inequality \eqref{eq:Holder} implies that 
$$
\int_0^{\pi}|\alpha_i| [u,u']d\theta\leq  \sqrt{A(u)SA(M_0)}.
$$
We conclude from equation \eqref{eq:integraalphaibetai} that $\beta_{i}$ is uniformly bounded. Similarly $\alpha_i$ uniformly bounded. 
\end{proof}

\begin{Proposition}\label{Prop:ConvAlphaBeta}
The functions $\alpha_i$ and $\beta_i$ are converging to $0$ in the $C^{\infty}$ topology.
\end{Proposition}

\begin{proof}
From lemma \ref{lemma:boundsalphaibetai}, the families $(\alpha_i)$ and $(\beta_i)$ are equicontinuous. By equation \eqref{eq:sumsquares}, 
$$
\lim \int_0^{\pi}\alpha_{i}^2[u,u']d\theta=0.
$$
This implies that $\lim\alpha_i=\lim\beta_i=0$ in the $C^{\infty}$-topology.
\end{proof}

By the above lemmas, the diameter of $M_i$ and $N_i$ are converging to zero and so $O=O(\gamma)$ is in fact a set consisting of a unique point. Theorem \ref{thm:ConvMiNi} is now an easy consequence
of proposition \ref{Prop:ConvAlphaBeta}.

% ------------------------------------------------------------------------
\end{document}